\documentclass[12pt]{amsart}

\usepackage{graphicx, amsmath, amssymb}

\newtheorem{thm}{Theorem}[section]
\newtheorem{prop}[thm]{Proposition}

\newtheorem{lemma}[thm]{Lemma}

\newtheorem{definition}[thm]{Definition}
\theoremstyle{definition}
\newtheorem{example}[thm]{Example}

\title[A special case of completion invariance]{A special case of completion invariance for the $c_2$ invariant of a graph.}

\author{Karen Yeats}

\address{Department of Combinatorics and Optimization\\University of Waterloo\\Waterloo, ON, Canada}
\email{kayeats@uwaterloo.ca}

\thanks{The author thanks Dmitry Doryn for discussions which inspired this work and the IBS Centre for Geometry and Physics in Korea for their hospitality in hosting those discussions as well as their financial support from IBS-R003-S1 for the visit.  The author thanks Iain Crump for many discussions about graphical interpretations of the $c_2$ invariant over the last few years.  The author also thanks the referees.  She is supported by an Alexander von Humboldt fellowship and an NSERC Discovery grant.}


\begin{document}
\begin{abstract}
  The $c_2$ invariant is an arithmetic graph invariant defined by Schnetz.  It is useful for understanding Feynman periods.  Brown and Schnetz conjectured that the $c_2$ invariant has a particular symmetry known as completion invariance.
This paper will prove completion invariance of the $c_2$ invariant in the case that we are over the field with 2 elements and the completed graph has an odd number of vertices.
  The methods involve enumerating certain edge bipartitions of graphs; two different constructions are needed.
\end{abstract}

\maketitle

\section{Introduction}
The $c_2$ invariant is an arithmetic graph invariant defined by Schnetz \cite{SFq}, see below, which is useful for understanding Feynman periods.  Feynman periods are a class of integrals defined from graphs which are simpler than but closely related to Feynman integrals.  Based on this connection and on computational evidence, there are certain symmetries which the $c_2$ invariant is believed to have.  A key such symmetry, known as completion invariance and defined below, was first conjectured by Brown and Schnetz in 2010 \cite{BrS} and has turned out to be quite difficult to prove.  The main result of this paper, Theorem~\ref{thm main} is the completion invariance of the $c_2$ invariant in the case that $p=2$ and the completed graph has an odd number of vertices.

\medskip

Throughout $K$ will be a connected 4-regular simple graph\footnote{For unexplained graph theory language or notation see \cite{Dbook}.}.  The result of removing any vertex of $K$ is called a \emph{decompletion} of $K$.  Different decompletions will typically be non-isomorphic and $K$ can be uniquely reconstructed from any of its decompletions.  We will say that $K$ is the \emph{completion} of any of its decompletions.  See Figure~\ref{fig decompl} for an example.

Suppose $G$ is a decompletion of $K$.  Then we can view $G$ as a Feynman diagram in $\phi^4$ theory.  For those not familiar with Feynman diagrams and quantum field theory, briefly: The edges of the graphs represent particles; the vertices particle interactions.  This particular quantum field theory has only a quartic interaction (that is the $4$ in $\phi^4$), and so all vertices must have degree 4.  The vertices of $G$ which are no longer 4-regular due to decompletion are taken to have additional external edges attached which represent particles entering or exiting the system.  Completing $G$ then means attaching all the external edges to a new vertex which one can think of as being ``at infinity'' hence corresponding to completion in the geometric sense.  This connection to geometric completion is easily made precise at the level of the Feynman period, see \cite{Sphi4}.  The Feynman integral of a Feynman diagram is the thing one actually wants for physics because it computes the contribution of that Feynman diagram to whatever physical process one is interested in.  For more on perturbative quantum field theory see \cite{iz}.

\begin{figure}
  \includegraphics{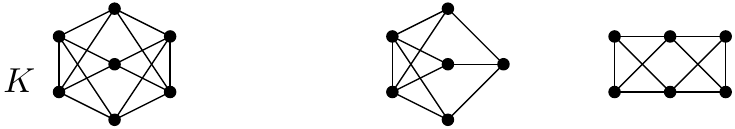}
  \caption{A graph $K$ with two non-isomorphic decompletions}\label{fig decompl}
\end{figure}

Given any graph $G$ (but most interestingly for the present purposes $G$ may be some decompletion of $K$), assign a variable $a_e$ to each edge $e \in E(G)$ and define the (dual) \emph{Kirchhoff polynomial} or first Symanzik polynomial to be
\[
\Psi_G = \sum_{T} \prod_{e\not\in T}a_e
\]
where the sum runs over all spanning trees of $G$.  For example, the Kirchhoff polynomial of a 3-cycle with edge variables $a_1$, $a_2$, $a_3$ is $a_1+a_2+a_3$ since removing any one edge of a cycle gives a spanning tree of the cycle.

Next define the \emph{Feynman period} to be
\[
\int_{a_e\geq 0} \frac{\prod da_e}{\Psi_G^2}\bigg|_{a_1=1}.
\]
This is an affine form of the integral; there is also a projective form, see \cite{Sphi4}.  This integral converges provided $G$ is at least internally 6-edge connected\footnote{That is, any way of removing fewer than 6 edges either leaves the graph connected or breaks the graph into two components one of which is a single vertex.}.  This corresponds to the Feynman diagram having no subdivergences because edge cuts with four or fewer edges, other than ones giving isolated vertices, represent sub-processes which yield divergent integrals.  There has been a lot of interest in the Feynman period because it is a sensible algebro-geometric, or even motivic, object \cite{bek,BrSinform,Brcosmic,Brbig,Mmotives,Sphi4}, but it is also a key piece of the Feynman integral.  It is a sort of coefficient of divergence for the Feynman integral and has the benefit of not depending on the many parameters and choices which the full Feynman integral depends on, so it is mathematically much tidier.  None-the-less it still captures some of the richness of Feynman integrals.  This is best illustrated by the number theoretic richness of the numbers which Feynman periods give, see \cite{bek,BrSinform,bkphi4,Snumbers, Sphi4}.

Given these connections to many deep and difficult things it should not be surprising that Feynman periods are still difficult to understand and compute, and so in \cite{SFq} Schnetz introduced the $c_2$ invariant in order to better understand the Feynman period.

\begin{definition}
  Let $p$ be a prime, let $\mathbb{F}_p$ be the finite field with $p$ elements, let $G$ be a connected graph with at least 3 vertices, and let $[\Psi_G]_p$ be the number of $\mathbb{F}_p$-rational points in the affine variety $\{\Psi_G=0\} \subseteq \mathbb{F}_p^{|E(G)|}$.  Then the \emph{$c_2$-invariant} of $G$ at $p$ is
  \[
  c_2^{(p)}(G) = \frac{[\Psi_G]_p}{p^2} \mod p.
  \]
\end{definition}
That this is well defined is proved in \cite{SFq}.  The same definition can be made for prime powers, but we will stick to primes $p$.  The $c_2$ invariant has interesting properties \cite{BrS,BrSY,Dc2,D4face}, yields interesting sequences in $p$ such as coefficient sequences of modular forms \cite{BrS3,Lmod}, and predicts properties of the Feynman period.  A simple and striking example of the last of these is that the $c_2^{(p)}(G)=0$ for all $p$ corresponds to when the Feynman period apparently has a drop in transcendental weight relative to the size of the graph\footnote{Given that proving transcendentality for single zeta values, much less multiple zeta values or other more exotic numbers appearing in Feynman periods, is completely out of reach, one must recast this more formally in order to make a precise statement or prove a precise result along these lines, see \cite{BrS,SFq}.}.  For example the two decompleted graphs in Figure~\ref{fig decompl} both have $c_2^{(p)}=0$ for all $p$ and both have period $36\zeta(3)^2$ which has weight $3+3=6$ while some other graphs with the same number of edges have weight $7$, see \cite{Sphi4}.

The $c_2$ invariant should have something to do with the Feynman period because both counting points and taking period integrals are controlled by the geometry of the variety of $\Psi_G$.  If the Feynman periods are nice then the point counts as a function of $p$ should be nice and vice versa.  More specifically, inspired by known Feynman periods at the time being multiple zeta values\footnote{This is no longer the case unless there are some outrageous identities in play -- as with transcendentality questions for the multiple zeta values themselves, proving an absence of relations is almost impossibly hard even when we have solid theoretical reasons to think there should be none.}, Kontsevich informally conjectured that $[\Psi_G]_p$ should be a polynomial in $p$.  This turned out to be very false \cite{BrBe}.  The $c_2$ invariant is one measure of whether or how badly Kontsevich's conjecture fails for a given graph; if it holds for that graph then the $c_2$ invariant is the quadratic coefficient of the polynomial, thus explaining the $2$ in $c_2$.

The $c_2$ invariant has or is believed to have all of the symmetries of the Feynman period.  In particular, it is known that the Feynman periods of two graphs with the same completion are the same, see \cite{Sphi4}.  The $c_2$ invariant was also conjectured by Brown and Schnetz in 2010 (Conjecture 4 of \cite{BrS}) to have this completion symmetry.  This is arguably the most interesting open problem on the arithmetic of Feynman periods.  Very little progress has been made.  We do know that if $G$ has two triangles sharing an edge then the question of completion invariance for $c_2^{(p)}(G)$ can be reduced to completion invariance of a particular smaller graph.  This is known as double-triangle reduction, see Corollary 34 of \cite{BrS}. 

The main result of this paper puts the first crack into the conjecture itself, proving it in the special case when $p=2$ and the completed graph has an odd number of vertices: 

\begin{thm}\label{thm main}
  Let $K$ be a connected 4-regular graph with an odd number of vertices.
  Let $v$ and $w$ be vertices of $K$.  Then $c_2^{(2)}(K-v) = c_2^{(2)}(K-w)$. 
\end{thm}

The basic approach takes a graph theoretic perspective on the $c_2$ invariant.  The same approach also underlies the $c_2$ results for families of graphs from \cite{CYgrid, Ycirc}, and the reader may like to look there for further examples.  The strength of the approach is that it is sufficiently different from a more algebraic or geometric approach to possibly make progress where other methods stalled.  The weakness of the approach is that extending beyond $p=2$ will be tricky, perhaps impossible.  See Section~\ref{sec conclusion} for further comments.  The vertex parity restriction is both more mysterious and perhaps more tractable.  Again see Section~\ref{sec conclusion}.  A different graph theoretic perspective on the $c_2$ invariant is given by Crump in \cite{Cphd}.

\medskip

The structure of the paper is as follows.  Section~\ref{sec background} defines the different graph polynomials that will be needed and collects together some important lemmas from other sources.  Section~\ref{sec reduction} then uses these lemmas to reduce the problem of proving the main theorem to a problem of determining the parity of the number of certain edge bipartitions.  The real work of proving Theorem~\ref{thm main} then comes in Sections~\ref{sec involutions} and \ref{sec compat cycles}.  The set of these edge bipartitions are divided into five pieces.  Two of the pieces are proved to have even size in Section~\ref{sec involutions} by finding fixed-point free involutions.  The other three pieces are proved to have even size in Section~\ref{sec compat cycles} by giving a cycle swapping rule to transform edge bipartitions.  Finally Section~\ref{sec conclusion} concludes with some comments about the result, the proof, and the way forward.

\section{Background}\label{sec background}

The first step is to define some additional graph polynomials that will be needed in order to prove the main result.

Let $G$ be a graph.  Choose an arbitrary order for the edges and the vertices of $G$ and choose an arbitrary orientation for the edges of $G$.  Let $E$
be the signed incidence matrix of $G$ (with the vertices corresponding to rows and the edges corresponding to columns) with one arbitrary row removed and let $\Lambda$ be the diagonal matrix with the edge variables of $G$
on the diagonal.  Let
\[
M = \begin{bmatrix} \Lambda & E^t \\ -E & 0 \end{bmatrix}.
\]
Then
\[
\det M = \Psi_G.
\]
This can be proved by expanding out the determinant, see \cite{Brbig} Proposition 21, or by using the Schur complement and the Cauchy-Binet formula, see \cite{VY}.  In both cases it comes down to the fact that the square full rank minors of
$E$ are $\pm 1$ for columns corresponding to
the edges of a spanning tree of $G$ which is the essence of the matrix tree theorem.  In this and other ways the matrix $M$ behaves much like the Laplacian matrix of a graph (with variables and one matching row and column removed) but the pieces which make it up are separated out, so call $M$ the \emph{expanded Laplacian} of $G$.

Assume that we have made a fixed choice of orders and orientation so as to define a fixed $M$ in all that follows.  Thus it will not matter whether we talk of edges or edge indices.  In particular 
we will use $G/e$ and $G\backslash e$ for the contraction and deletion respectively in the graph $G$ with the same meaning whether $e$ is an edge or an edge index. 

If $I$ and $J$
are sets of edge indices then let $M(I, J)$ be the matrix $M$ with the rows indexed by elements of $I$ removed and the columns indexed by elements of $J$ removed.  Brown in \cite{Brbig} defined the following Dodgson polynomials
\begin{definition}
Let $I$, $J$, and $K$ be sets of edge indices with $|I|=|J|$.  Define
\[
\Psi^{I,J}_{G,K} = \det M(I, J)|_{a_e=0 \text{ for } e\in K}.
\]
\end{definition}
We may leave out the $G$
subscript when the graph is clear and leave out the $K$ when $K= \varnothing$.

Dodgson polynomials have nice contraction-deletion properties and many interesting relations see \cite{Brbig}.  Making different choices in the construction of $M$ may change the overall sign of a Dodgson polynomial, but since we will be concerned with counting zeros of these polynomials the overall sign is of no concern.

Dodgson polynomials can also be expressed in terms of spanning forests; this is a consequence of the all minors matrix tree theorem \cite{Chai}.  The following spanning forest polynomials are convenient for this purpose.
\begin{definition}
  Let $P$ be a set partition of a subset of the vertices of $G$.  Define
  \[
  \Phi^P_G = \sum_{F}\prod_{e \not\in F}a_e
  \]
  where the sum runs over spanning forests $F$ of $G$ with a bijection between the trees of $F$ and
the parts of $P$
where each vertex in a part lies in its corresponding tree.
\end{definition}
Note that trees consisting of isolated vertices are permitted.  In most of the argument we will be working with spanning forest polynomials where $P$ has exactly 2 parts.  The corresponding spanning forests thus have exactly two parts and will be known as \emph{2-forests}.  For example, if $G$ is as in Figure~\ref{fig spanning for eg} then $\Phi_G^{\{v_1,v_2\},\{v_3\}} = (e+d)(ca+cb + ab + fb+gb)$

\begin{figure}
  \includegraphics{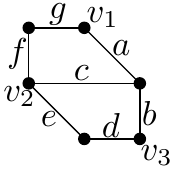}
  \caption{A graph for a spanning forest example.}\label{fig spanning for eg}
\end{figure}

The precise relationship between Dodgson polynomials and spanning forest polynomials is given in the following proposition.
\begin{prop}[Proposition  12  from  \cite{BrY}]
Let $I$,$J$, and $K$ be  sets  of  edge  indices  of
$G$
with $|I|=|J|$.  Then
\[
\Psi^{I,J}_{G,K} = \sum\pm \Phi^P_{G\backslash(I\cup J\cup K)}
\]
where the sum runs over all set partitions $P$ of the end points of edges of $(I \cup J\cup K) \backslash (I\cap J)$ such that all the forests corresponding to $P$ become spanning trees in both $G\backslash I/(J\cup K)$ and $G\backslash J/(I\cup K)$.
\end{prop}
The signs appearing in the proposition can be determined, see \cite{BrY}, however they are of no concern for the present since we will be working modulo 2.  Note that if an edge index is in both $I$ and $J$ then it is deleted in both $G\backslash I$ and in $G\backslash J$ and so cannot then be contracted; in the contraction simply ignore edge indices whose edges are no longer there.

The instances of this proposition which will be needed below also serve as a good example of these objects and how to manipulate them.
\begin{example}\label{eg psi to phi}
  Suppose we have a graph $G$ with a 3-valent vertex $u$ and let the edges incident to $u$ be $1$, $2$, and $3$, as in the first graph in Figure~\ref{fig eg G}.

  \begin{figure}
    \includegraphics{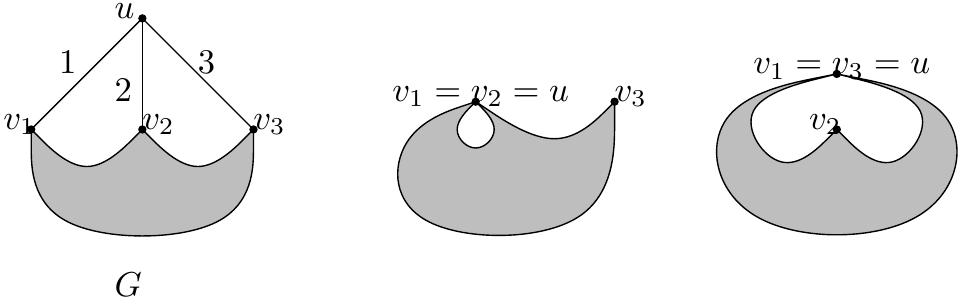}
    \caption{An example graph for some Dodgson polynomial to spanning forest polynomial computations and two of its minors.}\label{fig eg G}
  \end{figure}

  Consider the Dodgson polynomial $\Psi^{12, 13}_G$.  To find its expression in terms of spanning forest polynomials we need to look for sets of edges of $G\backslash 123$ which give a spanning tree in $G\backslash 12/3$ and in $G\backslash 13/2$.  These smaller graphs are each isomorphic to $G-u$ and so we obtain
  \[
  \Psi^{12,13}_G = \Psi_{G-u} = \Phi^{\{u\}, \{v_1,v_2,v_3\}}_{G\backslash 123} = \Phi^{\{v_1,v_2,v_3\}}_{G-u}.
  \]

  Consider the Dodgson polynomial $\Psi^{2,3}_{G,1}$.  To find its expression in terms of spanning forest polynomials we need to look for sets of edges of $G\backslash 123$ which give a spanning tree in $G\backslash 2/13$ and in $G\backslash 3/12$.  These two smaller graphs are the second and third graphs in Figure~\ref{fig eg G}.  If a set of edges is a forest for both of these graphs then on $G\backslash 123$ we must have $v_1$ and $v_2$ in different trees as well as $v_1$ and $v_3$ in different trees; $u$ must also be in a different tree from all the other vertices.  Furthermore, to get a tree in the two minors of Figure~\ref{fig eg G} the spanning forest of $G\backslash 123$ must have exactly three trees, including the tree consisting of the isolated vertex $u$ alone.  There is one partition of $\{u, v_1, v_2, v_3\}$ which satisfies these properties, namely $\{u\},\{v_1\},\{v_2,v_3\}$ and all forests compatible with this partition give trees in the two minors of Figure~\ref{fig eg G}.  Therefore
  \[
  \Psi^{2,3}_{G,1} = \Phi^{\{u\}, \{v_1\}, \{v_2,v_3\}}_{G\backslash 123}.
  \]
  Since the vertex $u$ is isolated in $G\backslash 123$, we can rewrite this without $u$ as
  \[
  \Psi^{2,3}_{G,1} = \Phi^{\{v_1\}, \{v_2,v_3\}}_{G-u}.
  \]
\end{example}

We can use the particular Dodgson polynomials from the example to compute the $c_2$ invariant more easily.  Continue for the next two lemmas with the notation $[F]_p$ for the number of $\mathbb{F}_p$-rational points on the affine variety defined by the polynomial $F$ reduced modulo $p$.  The polynomials we will deal with will always come from graphs and so the affine space in question will be of dimension the number of edges of the graph.

\begin{lemma}[Lemma 24 of \cite{BrS} along with inclusion-exclusion]\label{lem D3}
  Suppose $G$ is a graph with  $2 + |E(G)| \leq 2|V(G)|$.  Let $i,j,k$ be distinct edge indices of $G$ and let $p$ be a prime.  Then
  \[
  c_2^{(p)}(G) = -[\Psi^{ik,jk}_{G}\Psi^{i,j}_{G,k}]_p \mod p.
  \]
\end{lemma}

Note that decompletions of $4$-regular graphs satisfy the hypotheses of this lemma.  This lemma is useful because we no longer have to divide by $p^2$ but rather can obtain the $c_2$ invariant directly as a point count modulo $p$.

The last lemma we need is a corollary of one of the standard proofs of the Chevalley-Warning theorem, see section 2 of \cite{Ax}.

\begin{lemma}\label{lem cor of CW}
  Let $F$ be a polynomial of degree $N$ in $N$ variables, $x_1, \ldots, x_N$, with integer coefficients. The coefficient of $x_1^{p-1}\cdots x_N^{p-1}$ in $F^{p-1}$ is $[F]_p$ modulo $p$.
\end{lemma}

This last lemma is particularly useful when used in conjunction with the previous lemma.  Suppose $G$ is a decompletion of a 4-regular graph and $i,j,k$ are distinct edge indices.  Then $\Psi^{ik,jk}_{G}\Psi^{i,j}_{G,k}$ has both degree and number of variables equal to the number of edges of $G\backslash ijk$.  Thus $c_2^{(p)}(G)$ is equal to the coefficient of $x_1^{p-1}\cdots x_N^{p-1}$ in $(\Psi^{ik,jk}_{G}\Psi^{i,j}_{G,k})^{p-1}$ modulo $p$.  In view of this, we will no longer need to consider point counts and so that frees up square brackets for the usual algebraic combinatorics notation for the \emph{coefficient-of} operator.  Namely, the coefficient of a monomial $m$ in a polynomial $F$ is denoted $[m]F$.  Rewriting the previous observation in this notation we get
\[
  c_2^{(p)}(G) = [x_1^{p-1}\cdots x_N^{p-1}](\Psi^{ik,jk}_{G}\Psi^{i,j}_{G,k})^{p-1} \mod p
\]
where $N$ is the number of edges of $G\backslash ijk$.  If we further suppose that $i$, $j$, and $k$ meet at a $3$-valent vertex $u$ and their other ends are $v_1$, $v_2$, and $v_3$ (this case will be sufficient for our purposes) then we can incorporate the computations of Example~\ref{eg psi to phi} to get
\begin{equation}\label{eq key set up}
c_2^{(p)}(G) = [x_1^{p-1}\cdots x_N^{p-1}](\Phi^{\{v_1\}, \{v_2,v_3\}}_{G-u}\Psi_{G-u})^{p-1} \mod p.
\end{equation}
The polynomials do in general depend on the choice of $v_1$, though the coefficient in question modulo $p$ clearly does not.  We will use the freedom of choice of $v_1$ later.

When $p=2$ \eqref{eq key set up} has a particularly nice graphical interpretation which can be derived as follows.
We are interested in the parity of $[x_1\cdots x_N]\Phi^{\{v_1\}, \{v_2,v_3\}}_{G-u}\Psi_{G-u}$, but the coefficient of $x_1\cdots x_N$ in $\Phi^{\{v_1\}, \{v_2,v_3\}}_{G-u}\Psi_{G-u}$ is the number of ways to partition the variables into two monomials, one from $\Phi^{\{v_1\}, \{v_2,v_3\}}_{G-u}$ and one from $\Psi_{G-u}$, since both of these polynomials\footnote{The same idea works for more general products of Dodgson polynomials, but some care must be taken as Dodgson polynomials are in general signed sums of spanning forest polynomials.  For some examples computing in this way see \cite{CYgrid,Ycirc}.} have all monomials appearing with coefficient 1.  The variables correspond to edges, so this is equivalent to counting the number of ways to partition the edges of $G-u$  into two parts, one part when removed gives a spanning tree and the other part when removed gives a spanning 2-forest compatible with $\{v_1\}, \{v_2, v_3\}$.  Swapping the roles of the two parts this is equivalent to counting the number of ways to partition the edges of $G-u$ into two parts, one of which is a spanning tree and the other of which is a spanning 2-forest compatible with $\{v_1\}, \{v_2,v_3\}$.  The $c_2$ invariant at $p=2$ is simply the parity of this count, so we are determining $c_2$ by counting assignments of the edges of $G-u$ to particular spanning trees and forests.

For $p>2$ the same idea works, but we must assign $p-1$ copies of each edge among $p-1$ copies of each polynomial.  We can view this as partitioning the edges of the graph we obtain by taking $G-u$ and replacing each edge with $p-1$ edges in parallel.  There are many more possibilities when partitioning all these multiple edges into $2p-2$ polynomials; the practicalities of working with this approach for $p>2$ are daunting.

%\medskip
%
%Finally, we need to summarize what is known about the relation between double triangle reduction and the $c_2$ invariant.  Suppose a graph contains two triangles which share an edge as a subgraph.  Let $x$ and $y$ be the ends of the shared edge and $u$ and $v$ be the other two vertices of the triangles.  The graph resulting from removing the edges $ux$ and $vx$, contracting the edge $xy$, and adding a new edge from $u$ to $v$ is called \emph{double triangle reduction}.  See Figure~\ref{fig double triangle}.  If a graph $K'$ results from performing one or more double triangle reductions on a graph $K$ then $K'$ is known as a \emph{double triangle ancestor} of $K$.
%
%\begin{figure}\label{fig double triangle}
%  \includegraphics{DT}
%  \caption{Double triangle reduction.}
%\end{figure}
%
%The result is not there because the reductions at infinity haven't been done

\section{Reduction to a combinatorial counting problem}\label{sec reduction}

%The arguments in this section give the proof of Theorem~\ref{thm main}.  First we will give the set up which will hold throughout the section, then the main argument has two main parts each in their own subsection.

%\subsection{Set up}\label{subsec set up}

Take $K$ to be a (fixed) connected $4$-regular graph 
with an odd number of vertices
and take $v$, $w$ to be vertices of $K$.  Since $K$ is connected there is a path between any two vertices of $K$ and thus it suffices to prove Theorem~\ref{thm main} in the case that $v$ and $w$ are joined by an edge.  Thus add as an assumption that $v$ and $w$ are joined by an edge.

If $v$ and $w$ have three common neighbours then $K-v$ and $K-w$ are isomorphic so the result is trivial.  Therefore we can assume that $v$ and $w$ have at most two common neighbours.

The case when $v$ and $w$ have two common neighbours is special for two reasons.
First, when $v$ and $w$ have two common neighbours then there is a double triangle (two triangles sharing an edge) involving $v$, $w$, and their common neighbours.  The arguments of \cite{BrS} on double triangle invariance of $c_2$ have not been extended to the case where one of the vertices of the double triangle is the completion vertex and so these arguments cannot be used here.  However, the double triangle arguments should be readily generalizable to this situation and would work for all $p$, but the required techniques and setup are somewhat different and so this will not be pursued here.  Rather I will simply leave it as a comment that the case where $v$ and $w$ have exactly two common neighbours should be a consequence of the other cases because of the double triangle, and instead will prove this case directly with the present methods.

Second\footnote{Thanks to a referee for this observation.}, the case where $v$ and $w$ have two common neighbours is special enough that we can remove the requirement that $K$ have an odd number of vertices.  This argument is suggestive of how the parity restriction should hopefully be removable in general and is discussed in Section~\ref{sec conclusion}.

In the end we must consider $0$, $1$, or $2$ common neighbours for $v$ and $w$.  These three cases are illustrated in Figure~\ref{fig K}; label the vertices of $K$ as in the figure.

\begin{figure}
  \scalebox{0.8}{\raisebox{-2cm}{\includegraphics{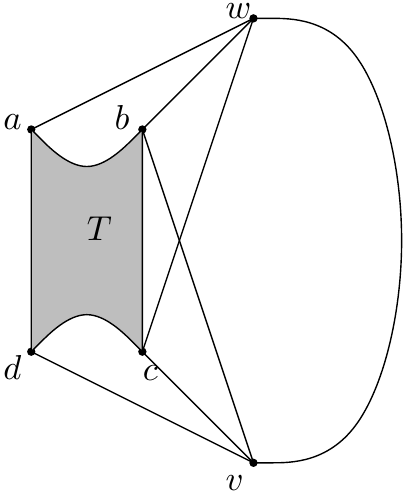}}} \text{ or }\scalebox{0.8}{\raisebox{-2cm}{\includegraphics{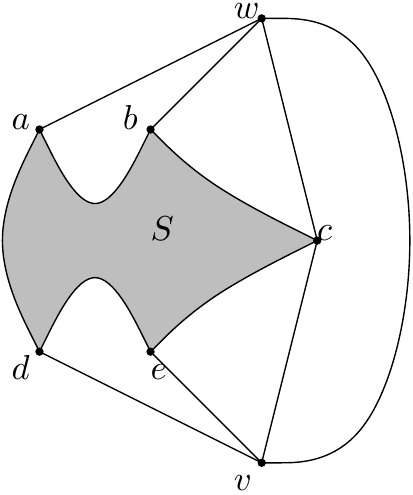}}} \text{ or } \scalebox{0.8}{\raisebox{-1.5cm}{\includegraphics{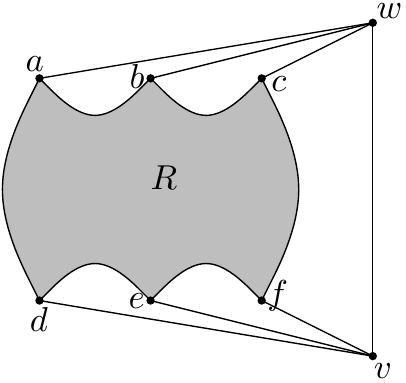}}}
  \caption{The graph $K$}\label{fig K}
\end{figure}

If $v$ and $w$ have no common neighbours, let $R= K-\{v,w\}$; that is $R$ is the grey blob on the right hand side of Figure~\ref{fig K}.
If $v$ and $w$ have one common neighbour, let $S= K-\{v,w\}$; that is $S$ is the grey blob in the middle of Figure~\ref{fig K}.
Finally, if $v$ and $w$ have two common neighbours, let $T=K-\{v,w\}$; that is $T$ is the grey blob on the left hand side of Figure~\ref{fig K}.

By the fact that $K$ is 4-regular and has an odd number of vertices, the number of edges of $K-\{v,w\}$ is always $-1 \mod 4$ and so for each of $R$, $S$, and $T$ we will use  $x_1, x_2, \ldots, x_{4k-1}$ for the edge variables.

The next step is to use the results of Section~\ref{sec background} to rewrite the $c_2^{(2)}$ in the $R$, $S$, and $T$ cases.

Call a set partition with two parts a \emph{bipartition} and use the following notation: 

\begin{definition}
  Suppose $P$ is a bipartition of a subset of the vertices of $R$.  Then let $\mathcal{R}_{P}$ be the set of bipartitions of the edges of $R$ such that one part is a spanning tree of $R$ and the other part is a spanning 2-forest where one tree of the 2-forest contains all vertices of the first part of $P$ and the other contains all vertices of the second part of $P$.  Furthermore, let $r_P = |R_P|$.

  Define $\mathcal{S}_P$ and $s_P$ similarly for a bipartition of a subset of the vertices of $S$ and define $\mathcal{T}_P$ and $t_P$ similarly for $T$.
\end{definition}

For example, if $R=K_{3,3}$ as shown in the first part of Figure~\ref{fig bipartition eg}, then one of the elements of $\mathcal{R}_{\{a,b\}, \{c\}}$ is marked by the thick and dotted lines in the second part of the figure.  
Permuting $a$ and $b$ and permuting $d$, $e$, and $f$ we get $12$ elements of  $\mathcal{R}_{\{a,b\}, \{c\}}$.  In this case one other form can occur namely where the isolated vertex $c$ is the tree containing $c$ in the spanning 2-forest.  There are 6 such elements and so $r_{\{a,b\}, \{c\}} = 18$ in this case.
%All of the elements of $\mathcal{R}_{\{a,b\}, \{c\}}$ in this case have the same form up to permuting $a$ and $b$ and permuting $d$, $e$, and $f$.  Thus $r_{\{a,b\}, \{c\}} = 12$ in this case.

\begin{figure}
  \includegraphics{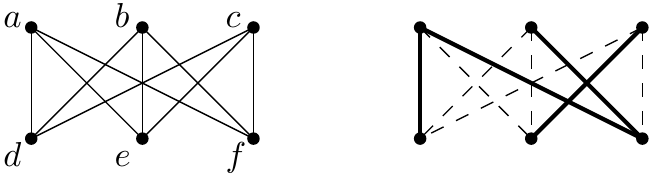}
  \caption{A graph and an edge bipartition of it compatible with $\{a,b\}, \{c\}$.}\label{fig bipartition eg}
\end{figure}
  
\begin{prop}
With notation as above, when $v$ and $w$ have no common neighbours
\begin{align}
  c_2^{(2)}(K-\{v\}) & = r_{\{a,b\}\{c\}} + r_{\{a,c\}, \{b\}} + r_{\{b,c\}, \{a\}} \mod 2, \label{eq c2r1} \\
  c_2^{(2)}(K-\{w\}) & = r_{\{d,e\}, \{f\}} + r_{\{d,f\}, \{e\}} + r_{\{e,f\}, \{d\}} \mod 2 \label{eq c2r2}
\end{align}
while when $v$ and $w$ have only neighbour $c$
\begin{align}
  c_2^{(2)}(K-\{v\}) & = s_{\{a,b\}, \{c\}} \mod 2, \label{eq c2s1} \\
  c_2^{(2)}(K-\{w\}) & = s_{\{d,e\}, \{c\}} \mod 2 \label{eq c2s2}
\end{align}
and when $v$ and $w$ have neighbours $b$ and $c$
\begin{align}
  c_2^{(2)}(K-\{v\}) & = t_{\{a\}, \{b,c\}} \mod 2, \label{eq c2t1} \\
  c_2^{(2)}(K-\{w\}) & = t_{\{d\}, \{b,c\}} \mod 2. \label{eq c2t2}
\end{align}
\end{prop}

\begin{proof}
By Lemmas~\ref{lem D3} and \ref{lem cor of CW} and Example~\ref{eg psi to phi}, as encapsulated in \eqref{eq key set up}, when $v$ and $w$ have no common neighbours
\begin{align*}
  c_2^{(2)}(K-\{v\}) & = [x_1x_2\cdots x_{4k-1}]\Phi_R^{\{a,b\},\{c\}}\Psi_R \mod 2 \\
  & = [x_1x_2\cdots x_{4k-1}]\Phi_R^{\{a,c\},\{b\}}\Psi_R \mod 2 \\
  & = [x_1x_2\cdots x_{4k-1}]\Phi_R^{\{b,c\},\{a\}}\Psi_R \mod 2
\end{align*}
and so by the joys of working modulo 2
\begin{align*}
  c_2^{(2)}(K-\{v\}) = \ & [x_1x_2\cdots x_{4k-1}]\Phi_R^{\{a,b\},\{c\}}\Psi_R \\& + [x_1x_2\cdots x_{4k-1}]\Phi_R^{\{a,c\},\{b\}}\Psi_R \\ &+ [x_1x_2\cdots x_{4k-1}]\Phi_R^{\{b,c\},\{a\}}\Psi_R \mod 2.
\end{align*}
Similarly
\begin{align*}
c_2^{(2)}(K-\{w\}) = \ & [x_1x_2\cdots x_{4k-1}]\Phi_R^{\{d,e\},\{f\}}\Psi_R \\ & + [x_1x_2\cdots x_{4k-1}]\Phi_R^{\{d,f\},\{e\}}\Psi_R \\ &+ [x_1x_2\cdots x_{4k-1}]\Phi_R^{\{e,f\},\{d\}}\Psi_R \mod 2. 
\end{align*}

Thus, by the edge assignment interpretation discussed at the end of Section~\ref{sec background}, $c_2^{(2)}(K-\{v\})$ is equal modulo 2 to the number of ways to partition the edges of $R$ into two parts where one part is a spanning tree and the other part is a spanning 2-forest where one tree of the 2-forest includes two vertices among $\{a,b,c\}$ and the other tree of the 2-forest includes the remaining vertex of $\{a,b,c\}$.  Similarly $c_2^{(2)}(K-\{w\})$ is equal modulo 2 to the number of ways to partition the edges of $R$ into two parts where one part is a spanning tree and the other part is a spanning 2-forest where one tree of the 2-forest includes two vertices among $\{d,e,f\}$ and the other tree of the 2-forest includes the remaining vertex of $\{d,e,f\}$.  Restated using our notation, this is the statement of the proposition when $v$ and $w$ have no common neighbours.

When $v$ and $w$ have only common neighbour $c$ we can calculate similarly.  We could again sum over the three possible bipartitions of $\{a,b,c\}$ and similarly for $\{c,d,e\}$.  However, the common neighbour $c$ breaks the symmetry and thus it will be more convenient for the remainder of the argument simply to work with the partitions $\{a,b\}, \{c\}$ and $\{c\},\{d,e\}$ giving the simpler result of the statement of the proposition in the case that $v$ and $w$ have common neighbour $c$.

Likewise when $v$ and $w$ have common neighbours $b$ and $c$, because the symmetry is broken it is most convenient only to take the partitions $\{a\}, \{b,c\}$ and $\{d\}, \{b,c\}$ and otherwise argue as above to obtain the result.
\end{proof}

In view of the above proposition, to prove Theorem~\ref{thm main}, it suffices to show that the parity of the number of edge partitions which contribute to the right hand side of \eqref{eq c2r1} is the same as the parity of the number  of edges partitions which contribute to the right hand side of \eqref{eq c2r2} and similarly for \eqref{eq c2s1} and \eqref{eq c2s2} and for \eqref{eq c2t1} and \eqref{eq c2t2}.

\begin{prop}\label{prop A expansion}
  With notation as above, when $v$ and $w$ have no common neighbours
  \begin{align*}
    c_2^{(2)}(K-\{v\}) - c_2^{(2)}(K-\{w\}) =\, &r_{\{a,b,d,e,f\},\{c\}} + r_{\{a,c,d,e,f\},\{b\}} \\
    & + r_{\{a\}, \{b,c,d,e,f\}} + r_{\{a,b\}, \{c,d,e,f\}} \\
    & + r_{\{a,c\}, \{b,d,e,f\}} + r_{\{a,d,e,f\}, \{b,c\}} \\
    & + r_{\{a,b,c,d,e\}, \{f\}} + r_{\{a,b,c,d,f\}, \{e\}} \\
    & + r_{\{a,b,c,e,f\}, \{d\}} + r_{\{a,b,c,f\}, \{d,e\}} \\
    & + r_{\{a,b,c,e\}, \{d,f\}} + r_{\{a,b,c,d\}, \{e,f\}} \mod 2
  \end{align*}
  while when $v$ and $w$ have only neighbour $c$
  \begin{align*}
    c_2^{(2)}(K-\{v\}) - c_2^{(2)}(K-\{w\}) =\, & s_{\{a,b,d\},\{c,e\}} + s_{\{a,b,e\}, \{c,d\}} \\ &+ s_{\{a,b\},\{c,d,e\}} + s_{\{a,d,e\}, \{b,c\}} \\&+ s_{\{a,c\}, \{b,d,e\}} + s_{\{a,b,c\}, \{d,e\}} \mod 2
  \end{align*}
  and when $v$ and $w$ have neighbour $b$ and $c$
  \[
  c_2^{(2)}(K-\{v\}) - c_2^{(2)}(K-\{w\}) = t_{\{a\}, \{b,c,d\}} + t_{\{a,b,c\},\{d\}} \mod 2.
  \]
\end{prop}

\begin{proof}
  Consider the right hand sides of \eqref{eq c2r1} and \eqref{eq c2r2}.
Enumerating all possibilities
\begin{align*}
  r_{\{a,b\}, \{c\}} =\ & r_{\{a,b,d,e,f\},\{c\}} + r_{\{a,b,d,e\}, \{c,f\}} + r_{\{a,b,d,f\}, \{c,e\}} + r_{\{a,b,e,f\}, \{c,d\}} \\
  & + r_{\{a,b,d\}, \{c,e,f\}}  + r_{\{a,b,e\}, \{c,d,f\}} + r_{\{a,b,f\}, \{c,d,e\}} + r_{\{a,b\}, \{c,d,e,f\}}
\end{align*}
and similarly for the other terms.  Collecting these calculations together, simplifying modulo 2, and performing analogous calculations with regards to the right hand sides of \eqref{eq c2s1}, \eqref{eq c2s2}, \eqref{eq c2t1}, and \eqref{eq c2t2} we get the proposition.
\end{proof}

It is best, in my view, to keep a graphical viewpoint with the above formulas.  For example, one can represent a given term by drawing the graph and marking the partition by using different vertex shapes.  Then the second equation of the statement of Proposition~\ref{prop A expansion} looks like
\begin{align*}
  & c_2^{(2)}(K-\{v\}) - c_2^{(2)}(K-\{w\}) \\ &= \raisebox{-.2cm}{\includegraphics{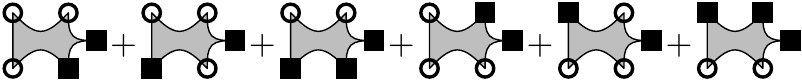}} \mod 2.
  \end{align*}
  The reader is encouraged to translate all the equations from here on out into this notation in order to better see the intuition behind the argument.
Note that this graphical notation is not the same as in \cite{BrY, CYgrid, Ycirc} where a graph with a marked partition in this way represented the spanning forest polynomial for that graph and partition.

One's first thought for proceeding from here would likely be to find a fixed-point-free involution of the union of the $\mathcal{R}$, $\mathcal{S}$, or $\mathcal{T}$ sets appearing (via their counts) in Proposition~\ref{prop A expansion}.  It is not clear how to do this for all the $\mathcal{R}$, $\mathcal{S}$ and $\mathcal{T}$ sets and so the method will be more complicated.    First we will deal with some of the $\mathcal{R}$ and $\mathcal{S}$ sets via involutions.  For the sets that remain we will build auxiliary graphs which use the properties of certain cycles along with the parity hypotheses on $K$ to show these remaining sets all have an even contribution.

\section{Some involutions from swapping around particular vertices}\label{sec involutions}

The involution is simplest in the case where $v$ and $w$ have only common neighbour $c$ and the partition breaks $a,b,c,d,e$ into $\{a,b,c\},\{d,e\}$ or $\{a,b\},\{c,d,e\}$, so we will begin by discussing that case.

\begin{lemma}\label{lem swap}
Let $\sigma$ be a bipartition of the edges of $S$ so that each part is a spanning forest (either or both of which may potentially be a spanning forest with one tree, that is a spanning tree) where in each forest every tree contains at least one of $a$, $b$, $d$, or $e$.  Then of the two edges incident to $c$ exactly one is in each part of $\sigma$ and swapping which part these two edges are in yields a new partition $\sigma'$ with the above listed properties, but not necessarily  partitioning the vertices among the trees in the same way.
\end{lemma}

\begin{proof}
  Note that $c$ is 2-valent in $S$.  Let $y$ and $z$ be the two neighbours of $c$. 

  By hypothesis, in each part of $\sigma$ the vertex $c$ is connected to some other vertices of $S$ since every tree contains at least one of $a$, $b$, $d$, or $e$ and so at least one edge incident to $c$ is assigned to this part.  Since $c$ is 2-valent this means that exactly one incident edge to $c$ is in each part of $\sigma$.  Consider the spanning forest $F$ corresponding to one part of $\sigma$.  In $F$, $c$ is a leaf and so removing the edge incident to $c$ isolates $c$ and does not otherwise change the connectivity of the trees in $F$.  Without loss of generality say that $y$ was the other end of this edge.  Vertex $z$ is in some tree of $F$ and adding the edge between $z$ and $c$ reconnects $c$ to one of the trees of $F$ while maintaining a forest structure.  The same holds for the other part of $\sigma$.
\end{proof}

\begin{lemma}\label{lem B vert swap}
There is a fixed-point free involution on $\mathcal{S}_{\{a,b,c\}, \{d,e\}} \cup \mathcal{S}_{\{a,b\}, \{c,d,e\}}$ and thus $s_{\{a,b,c\}, \{d,e\}} + s_{\{a,b\}, \{c,d,e\}} = 0 \mod 2$.
\end{lemma}

\begin{proof}
All edge bipartitions in $\mathcal{S}_{\{a,b,c\}, \{d,e\}} \cup \mathcal{S}_{\{a,b\}, \{c,d,e\}}$ satisfy the hypotheses of Lemma~\ref{lem swap}.  To each such an edge partition swap the parts to which the edges incident to $c$ belong.  Consider an edge partition coming from $\mathcal{S}_{\{a,b,c\}, \{d,e\}}$.  Removing the edges incident to $c$ disconnects $c$ in both the 2-forest and the tree.  Putting the edges back in, but in opposite parts of the bipartition, reconnects $c$ to one of the trees of the 2-forest as well as to the single tree.  If $c$ is reconnected to the tree of the 2-forest including $\{a,b\}$ then we now have another edge partition in $\mathcal{S}_{\{a,b,c\}, \{d,e\}}$.  If $c$ is now connected to the tree of the 2-forest involving $\{d,e\}$ then we now have an edge partition in $\mathcal{S}_{\{a,b\}, \{c,d,e\}}$.  The edge partition must be distinct from the initial partition since which edge incident to $c$ corresponds to the tree has changed.  This operation is clearly an involution.  Thus we get a fixed-point free involution on $\mathcal{S}_{\{a,b,c\}, \{d,e\}} \cup \mathcal{S}_{\{a,b\}, \{c,d,e\}}$.  Consequently, the size of this set is even.
\end{proof}

We can also apply the swapping map described in the previous proof to the other $\mathcal{S}$ sets.  However an edge partition in $\mathcal{S}_{\{a,b,d\}, \{c,e\}}$ can either be mapped to another partition from the same set or it can be mapped to $\mathcal{S}_{\{a,b,c,d\},\{e\}}$ which is not an $\mathcal{S}$ set appearing in Proposition~\ref{prop A expansion}.  An analogous situation occurs for the other $\mathcal{S}$ sets not dealt with in the previous lemma.  For these remaining $\mathcal{S}$ sets, instead use the methods of Section~\ref{sec compat cycles}.

\medskip

Next we look at some of the $\mathcal{R}$ sets which can be tackled with a generalized version of the above argument.

\begin{lemma}\label{lem controlV}
  Let $p_1\cup p_2$ be a partition of $\{a,b,c,d,e,f\}$ where $p_1$ consists of either all of $\{a,b,c\}$ and exactly one of $\{d,e,f\}$ or all of $\{d,e,f\}$ and exactly one of $\{a,b,c\}$.  Let $x$ be the element of $p_1$ which is alone from its trio.  Consider any edge partition $\tau$ in $\mathcal{R}_{p_1,p_2}$.  Let $t$ be the tree of the 2-forest of $\tau$ associated to $p_1$.

  There is a unique vertex $y$ with the following properties
  \begin{itemize}
  \item Either $y\in p_1$ and is 2-valent in $t$ or $y\not\in p_1$ and $y$ is 3-valent in $t$.
  \item Removing $y$ from $t$ gives a component containing exactly two vertices of $p_1$ and a component containing exactly one vertex of $p_1$ (hence the third component, if it exists also contains exactly one vertex of $p_1$).
  \item Either $x=y$ or $x$ is in one of the components after removing $y$ which contains exactly one vertex of $p_1$.
  \end{itemize}
\end{lemma}

We will call the vertex $y$ defined in the above lemma the \emph{control vertex} of $\tau$ and we call the vertex $x$ the \emph{outsider vertex}.

\begin{proof}
  Since the spanning tree of $\tau$ spans, every vertex of $t$ is at most 3-valent and the vertices of $p_1$ are at most 2-valent in $t$.  The union of the paths in $t$ between the vertices of $p_1$ gives a subtree of $t$ where every leaf is in $p_1$.  There are only finitely many configurations for this subtree; these are illustrated in Figure~\ref{fig control vert} where the edges in the figure represent paths in $t$.  For each configuration the three properties and uniqueness can be checked directly, remembering that the degree of $y$ remains the same in $t$ as in the subtree.
\end{proof}

\begin{figure}
  \includegraphics{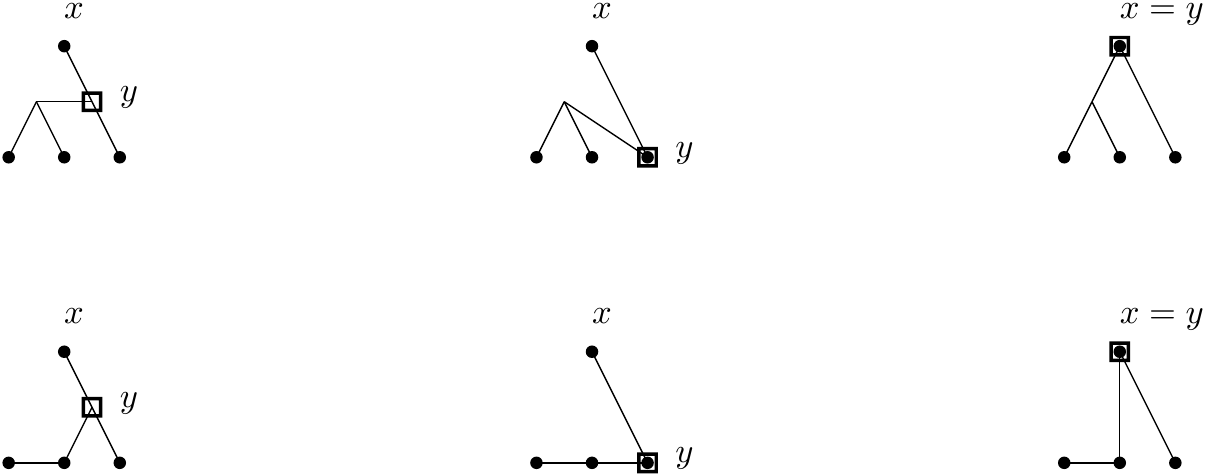}
  \caption{Schematic determining the control vertex $y$.  The lower three vertices in each schematic are $p_1-\{x\}$ in any order.}\label{fig control vert}
\end{figure}

Note that all the different configurations in Figure~\ref{fig control vert} can be viewed as special cases of the top left configuration where some of the paths have been contracted.

\begin{lemma}\label{lem A vert swap}
  There is a fixed-point free involution on
  \begin{gather*}
  \mathcal{R}_{\{a,b\}, \{c,d,e,f\}} \cup \mathcal{R}_{\{a,c\}, \{b,d,e,f\}} \cup \mathcal{R}_{\{a,d,e,f\}, \{b,c\}} \cup \mathcal{R}_{\{a,b,c,d\}, \{e,f\}} \\ \cup \mathcal{R}_{\{a,b,c,e\}, \{d,f\}} \cup \mathcal{R}_{\{a,b,c,f\}, \{d,e\}}
  \end{gather*}
  and thus
  \begin{align*}
  r_{\{a,b\}, \{c,d,e,f\}} + r_{\{a,c\}, \{b,d,e,f\}} + r_{\{a,d,e,f\}, \{b,c\}} & \\+ r_{\{a,b,c,d\}, \{e,f\}} + r_{\{a,b,c,e\}, \{d,f\}} + r_{\{a,b,c,f\}, \{d,e\}} & = 0 \mod 2.
  \end{align*}
\end{lemma}
  
\begin{proof}
  Let $\tau$ be an edge bipartition in the union of $\mathcal{R}$ sets in the statement of the lemma.  We need to set up some notation to build the involution
  \begin{itemize}
  \item All of these $\mathcal{R}$ sets satisfy the hypotheses of Lemma~\ref{lem controlV} and so let $y$ be the control vertex of $\tau$.
  \item There is exactly one edge incident to $y$ in the spanning tree of $\tau$.  Let $\epsilon$ be that edge.
  \item Let $z$ be the end of $\epsilon$ which is not $y$.
  \item Let $t_1$ be the tree of the 2-forest of $\tau$ which contains $y$ and let $t_2$ be the other tree of the 2-forest.
  %\item Let $f$ be the edge incident to $y$ which leads to the component of $t_1-\{y\}$ with two vertices from $\{a,b,c,d,e,f\}$.
  \end{itemize}
  Now build $\tau'$ from $\tau$ as follows.
  \begin{itemize}
  \item If $z \in t_2$ then let $\eta$ be the edge incident to $y$ which leads to the component of $t_1-\{y\}$ with two vertices from $\{a,b,c,d,e,f\}$.  Swap which part of the bipartition $\tau$ contains $\epsilon$ and which contains $\eta$ to obtain $\tau'$.
  \item If $z \in t_1$ then let $\eta$ be the edge incident to $y$ which leads to the component of $t_1-\{y\}$ which contains $z$.  Swap which part of the bipartition $\tau$ contains $\epsilon$ and which contains $\eta$ to obtain $\tau'$.
  \end{itemize}

  First lets check that $\tau'$ is in the union of $\mathcal{R}$ sets in the statement.   Similarly to the proof of Lemma~\ref{lem swap}, $y$ is a leaf in the spanning tree of $\tau$ so removing $\epsilon$ disconnects $y$ from the spanning tree and adding $\eta$ reconnects $y$ maintaining a spanning tree structure.  Removing $\eta$ further disconnects the 2-forest of $\tau$ into three components.  If we constructed $\tau'$ by the second case, then adding $\epsilon$ to the 2-forest reconnects the same components that were disconnected by the removal of $\eta$.  If we constructed $\tau'$ by the first case, then removing $\eta$ cuts off the component of $t_1-\{y\}$ containing exactly two vertices from $\{a,b,c,d,e,f\}$ from the rest of $t_1$; adding $\epsilon$ reconnected $t_2$ instead.  Furthermore, the outsider vertex is in the part of $t_1$ that gets connected with $t_2$.  The result is an edge partition which corresponds to a partition of $\{a,b,c,d,e,f\}$ satisfying the hypotheses of Lemma~\ref{lem controlV} and so is in one of the $\mathcal{R}$ sets in the statement.

  Next note that the map $\tau \mapsto \tau'$ is fixed-point free since which edge of the control vertex is in the spanning tree changes.

  Finally the control vertex of $\tau$ and $\tau'$ are the same and so applying the map twice is the identity.  Thus we get a fixed-point free involution on the union of $\mathcal{R}$ sets in the statement and so the size of this set is even.
\end{proof}

Note that in the case that $v$ and $w$ had common neighbour $c$, then $c$ always plays the role of the control vertex but compared to the configurations in Figure~\ref{fig control vert} more of the paths have been contracted away.  In this way the case with no common neighbours generalized the simpler argument in the common neighbour case.

\section{Compatible cycles}\label{sec compat cycles}

This section defines certain special cycles and investigates their properties.  These cycles will let us determine the parity of the remaining $r_P$ and $s_P$ in Proposition~\ref{prop A expansion}  as well as the $t_P$.

\begin{definition}
  \mbox{}
  
  \begin{enumerate}
  \item Call a bipartition of the edges of any graph (for our purposes either $R$, $S$ or $T$) such that one part gives a spanning tree and the other part gives a spanning 2-forest a \emph{valid} edge partition.
  \item
    Suppose we have a valid edge partition.  This gives a bipartition of all of the vertices of the graph according to which tree of the 2-forest they are in.  Call a cycle $C$ \emph{compatible} with the edge partition if all vertices of $C$ are in the same part of the vertex partition and exactly one edge of $C$ is in the part of the edge partition corresponding to the spanning tree.
  \end{enumerate}
\end{definition}

\begin{lemma}\label{lem count compatible}
Suppose we have a valid edge partition.  The number of compatible cycles is the same as the number of edges of the graph which are in the part of the edge partition corresponding to the spanning tree but where both ends of the edge are in the same tree of the 2-forest.
\end{lemma}

\begin{proof}
  Adding an edge joining two vertices of a tree gives a graph with a unique cycle.  When this fact is applied to one of the trees of the 2-forest then this cycle is compatible and every compatible cycle has this form.
\end{proof}

\begin{lemma}\label{lem odd and odd}
  Suppose we have a valid edge partition and let $V_1, V_2$ be the associated vertex partition.  Suppose that $\sum_{v\in V_i} \deg v$ is odd for $i=1,2$ and the total number of vertices of the graph is odd.  Then the number of compatible cycles is odd. 
\end{lemma}

\begin{proof}
  Let $\ell$ be the number of edges crossing the vertex partition and Let $e_i$ be the number of edges not in the 2-forest but with both ends in the tree of $V_i$ for $i=1,2$.  The number of edges leaving the tree of the 2-forest associated to $V_i$ is
  \[
  \sum_{v\in V_i}\deg v - 2(|V_i|-1) - 2e_i = \ell
  \]
  for $i=1,2$.
  Therefore, by the degree hypothesis, $\ell$ is odd.
  The number of edges of the spanning tree of the valid edge partition is
  \[
  |V_1|+|V_2|-1 = e_1+e_2+\ell.
  \]
  $|V_1|+|V_2|$ is the total number of vertices of the graph and so by hypothesis is odd.  Therefore $e_1+e_2$ is also odd.  By Lemma~\ref{lem count compatible}  the number of compatible cycles is $e_1+e_2$ and hence is odd, as desired.
\end{proof} 

Now we want to apply this lemma to the $\mathcal{R}$ and $\mathcal{S}$ of Proposition~\ref{prop A expansion} sets which were not dealt with in the previous section as well as the $\mathcal{T}$ sets of Proposition~\ref{prop A expansion}.

\begin{lemma}\label{lem the odds we need}
  Every edge partition of
  \begin{gather*}\mathcal{R}_{\{a,b,d,e,f\}, \{c\}}\cup \mathcal{R}_{\{a,c,d,e,f\}, \{b\}} \cup \mathcal{R}_{\{a\}, \{b,c,d,e,f\}} \cup \mathcal{R}_{\{a,b,c,d,e\}, \{f\}}\\ \cup \mathcal{R}_{\{a,b,c,d,f\}, \{e\}} \cup \mathcal{R}_{\{a,b,c,e,f\}, \{d\}} \\ \cup \mathcal{S}_{\{a,b,d\}, \{c,e\}} \cup \mathcal{S}_{\{a,b,e\}, \{c,d\}} \cup \mathcal{S}_{\{a,c\}, \{b,d,e\}} \cup \mathcal{S}_{\{b,c\}, \{a,d,e\}}\\\cup \mathcal{T}_{\{a\}, \{b,c,d\}} \cup \mathcal{T}_{\{a,b,c\}, \{d\}}
  \end{gather*}
  has an odd number of compatible cycles.
\end{lemma}

\begin{proof}
  It suffices to check the hypotheses of Lemma~\ref{lem odd and odd}.
  
  The number of vertices of each of $R$, $S$, and $T$ is odd by our running assumptions on $K$.
  
  Take any edge partition in the union above.  All the vertices other than $\{a,b,c,d,e,f\}$ are degree 4, so it suffices to check that the sums of the degrees of the vertices in each part of the defining partition of $\{a,b,c,d\}$, or $\{a,b,c,d,e\}$, or $\{a,b,c,d,e,f\}$ are odd.  For the $\mathcal{R}$ sets one part has degree sum $5\cdot 3$ and the other has degree sum $3$,  for the $\mathcal{S}$ sets one part has degree sum $3\cdot 3$ and the other has degree sum $3+2$, and for the $\mathcal{T}$ sets one part has degree sum $3$ and the other has degree sum $3+2+2$.  All of these are odd.
\end{proof}

\begin{lemma}\label{lem cycle swap}
  Suppose we have a valid edge partition and a compatible cycle $C$.
  Let $f$ be the one edge of $C$ not in the 2-forest.  If we were to remove $f$ the spanning tree would split into exactly two trees, call them $t_1$ and $t_2$.  There are a non-zero even number of edges of $C$ with one end in $t_1$ and the other end in $t_2$ (including $f$ as one of the possibilities).  If we take any such edge other than $f$ and swap it with $f$ in the edge partition
  then we obtain a valid edge partition corresponding to the same vertex partition.  
\end{lemma}

\begin{proof}
By construction $f$ has one end in $t_1$ and one end in $t_2$.  The vertices of $C$ can be bipartitioned based on whether they are in $t_1$ or in $t_2$.  Running around $C$ we must change which part of the bipartition we are in an even number of times in order to return to where we started, giving a non-zero even number of edges of the type described in the statement.

Let $f'\neq f$ be another edge of $C$ where we change from $t_1$ to $t_2$.  Removing $f$ from the spanning tree disconnects it into $t_1$ and $t_2$ while adding $f'$ reconnects $t_1$ and $t_2$ to obtain a spanning tree again.  Adding $f$ to the 2-forest creates one cycle, specifically $C$.  Removing any edge of $C$, in particular $f'$, returns us to a 2-forest with the same vertex partition. 
\end{proof}

\begin{lemma}\label{lem do cycle swap}
  \mbox{}

  \begin{enumerate}
  \item $r_{\{a,b,d,e,f\}, \{c\}} + r_{\{a,c,d,e,f\}, \{b\}} + r_{\{a\}, \{b,c,d,e,f\}} + r_{\{a,b,c,d,e\}, \{f\}} + r_{\{a,b,c,d,f\}, \{e\}} + r_{\{a,b,c,e,f\}, \{d\}} = 0 \mod 2$.
  \item $s_{\{a,b,d\}, \{c,e\}} + s_{\{a,b,e\}, \{c,d\}} + s_{\{a,c\}, \{b,d,e\}} + s_{\{b,c\}, \{a,d,e\}}  =  0 \mod 2$.
  \item $t_{\{a\}, \{b,c,d\}} + t_{\{a,b,c\}, \{d\}} = 0 \mod 2$.
  \end{enumerate}
\end{lemma}

\begin{proof}
  The construction is the same for all three cases cases.  It will be described explicitly in the $\mathcal{R}$ case.  Construct a graph $X_R$ as follows.  The vertices of $X_R$ are the edge partitions in $\mathcal{R}_{\{a,b,d,e,f\}, \{c\}}\cup \mathcal{R}_{\{a,c,d,e,f\}, \{b\}} \cup \mathcal{R}_{\{a\}, \{b,c,d,e,f\}} \cup \mathcal{R}_{\{a,b,c,d,e\}, \{f\}} \cup \mathcal{R}_{\{a,b,c,d,f\}, \{e\}} \cup \mathcal{R}_{\{a,b,c,e,f\}, \{d\}}$.  Two vertices of $X_R$ are adjacent if they are related by a swap as given in Lemma~\ref{lem cycle swap}.  Note that if a given edge assignment goes to another via such a swap then the second also goes to the first by such a swap since the cycle is compatible for either edge partition and the $t_1$, $t_2$ partition (in the notation of the proof of Lemma~\ref{lem cycle swap}) is also the same for both edge partitions.

  By Lemma~\ref{lem the odds we need}, for any vertex $x$ in $X_R$ there are an odd number of cycles which can yield swaps corresponding to edges incident to $x$.  Distinct cycles must give distinct swaps hence distinct edges.
  By Lemma~\ref{lem cycle swap}, each one of these cycles gives an odd number of edges incident to $x$ (one for each edge of the type described in the statement of Lemma~\ref{lem cycle swap} other than $f$ itself).  All edges incident to $x$ are obtained in this way so $x$ has odd degree.  This is true for all vertices of $X_R$, but by basic counting any graph has an even number of vertices of odd degree, so $X_R$ has an even number of vertices.

  Therefore the union of $\mathcal{R}$ sets defining $X_R$ has even size which is the first statement of the lemma.

  The argument for $X_S$ and $X_T$ is analogous using the edge partitions in $\mathcal{S}_{\{a,b,d\}, \{c,e\}} \cup \mathcal{S}_{\{a,b,e\}, \{c,d\}} \cup \mathcal{S}_{\{a,c\}, \{b,d,e\}} \cup \mathcal{S}_{\{b,c\}, \{a,d,e\}}$ and $\mathcal{T}_{\{a\}, \{b,c,d\}} \cup \mathcal{T}_{\{a,b,c\}, \{d\}}$ respectively.
\end{proof}

Note that the construction of $X_R$, $X_S$, and $X_T$ is closely related to the spanning tree graph (often just called the \emph{tree graph}, see \cite{CtreeG}) construction.  The vertices of the tree graph of a graph $G$ are the spanning trees of $G$ and two vertices of the tree graph are joined by an edge if the two spanning trees differ by removing one edge and replacing it with another.

\medskip

This is all we need to prove the main theorem.
\begin{proof}[Proof of Theorem~\ref{thm main}]
  As discussed at the beginning of section~\ref{sec reduction} it suffices to consider $v$ and $w$ joined by an edge and with zero, one, or two common neighbours.
  
  By Proposition~\ref{prop A expansion} we need only check that the parity of a certain sum of $r_{p_1,p_2}$ is even in the case that $v$ and $w$ have no common neighbours or a certain sum of $s_{p_1,p_2}$ is even in the case that $v$ and $w$ have one common neighbour, or a certain sum of $t_{p_1, p_2}$ is even in the case that $v$ and $w$ have two common neighbours.

  In the case that $v$ and $w$ have two common neighbours, the third part of Lemma~\ref{lem do cycle swap} gives that the required sum is even.
  
  In the case that $v$ and $w$ have one common neighbour, Lemma~\ref{lem B vert swap} and the second part of Lemma~\ref{lem do cycle swap} give that the required sum is even.
  
  In the case that $v$ and $w$ have no common neighbours, Lemma~\ref{lem A vert swap} and the first part of Lemma~\ref{lem do cycle swap} give that the required sum is even.
\end{proof}

\section{Discussion}\label{sec conclusion}

It should be the case that Theorem~\ref{thm main} is true for all $p$ and without the restriction that $K$ has an odd number of vertices, see conjecture 4 of \cite{BrS}.  This would be an excellent conjecture to prove because it would support the deep connection between the $c_2$ invariant and the Feynman period.  It is a surprisingly difficult and very interesting conjecture.

The restriction to $p=2$ came about because Lemma~\ref{lem cor of CW} is much simplified in the $p=2$ case.  For higher values of $p$ there is a still an edge assignment interpretation as discussed at the end of Section~\ref{sec background}, but each edge must be assigned $p-1$ times to build $p-1$ spanning trees and $p-1$ spanning 2-forests.  This greatly increases the complexity.  However, in principle these ideas may be extendable to $p>2$.  The practicalities would certainly be very hairy.  The question is whether or not the practicalities would be so hairy as to render the approach unworkable.

The vertex parity condition is more mysterious.  Note that it is only needed for the arguments of Section~\ref{sec compat cycles} not for the arguments of Section~\ref{sec involutions}.   If the number of vertices of $K$ were even then the required degree sum parity in Lemma~\ref{lem odd and odd} would also need to change to preserve the total number of compatible cycles being odd.  This translates into the methods of Section~\ref{sec compat cycles} applying to the other $\mathcal{R}$, $\mathcal{S}$, and $\mathcal{T}$ sets, namely the ones we already know how to tackle by Section~\ref{sec involutions}.  So if $K$ had an even number of vertices then the terms which in the odd case matched by swapping around the control vertex would become the terms which are even by compatible cycles, while the terms which used be even by compatible cycles would need a new argument.  The first place to look for this new argument would be as a generalized control vertex argument, however it is not clear how to do it. There is no obvious obstruction, rather the required construction is simply not apparent and so some further cleverness is required to progress.

In the $T$ case, that is the double triangle case where $v$ and $w$ have two common neighbours, things are sufficiently special that we can extend the argument to remove the parity condition.  This was graciously pointed out by a referee.  This argument manages to succeed by running through additional $\mathcal{T}$ sets which do not appear directly in the expression for $c_2$ but which occur while swapping.  I had tried similar things in other cases without finding a path to the result, but the approach remains promising as this argument shows.  Consider $\tau \in \mathcal{T}_{\{a\}, \{b,c,d\}}\cup\mathcal{T}_{\{a,b\}, \{c,d\}}$.  The vertex $b$ is 2-valent and one of its incident edges is in each of the tree and the 2-forest of $\tau$ since in both cases $b$ must connect to other vertices.  Thus by the same argument as Lemma~\ref{lem swap} we can swap the edges around $b$ and obtain a fixed-point free involution on $\mathcal{T}_{\{a\}, \{b,c,d\}}\cup\mathcal{T}_{\{a,b\}, \{c,d\}}$.  Therefore 
\[
t_{\{a\}, \{b,c,d\}} + t_{\{a,b\}, \{c,d\}} = 0 \mod 2.
\]
Symmetrically by swapping around vertex $c$ in  $\mathcal{T}_{\{a,b,c\}, \{d\}}\cup\mathcal{T}_{\{a,b\}, \{c,d\}}$ we get
\[
t_{\{a,b,c\}, \{d\}} + t_{\{a,b\}, \{c,d\}} = 0 \mod 2.
\]
Adding the two equations along with Proposition~\ref{prop A expansion} we get
\begin{align*}
c_2^{(2)}(K-\{v\}) - c_2^{(2)}(K-\{w\}) & = t_{\{a\}, \{b,c,d\}} + t_{\{a,b,c\},\{d\}} \mod 2 \\
& = 0 \mod 2.
\end{align*}

The statement and proof of Theorem~\ref{thm main} arose out of discussions with Dmitry Doryn  about looking for special cases where we could hope to progress on understanding the $c_2$ invariant.  The first version of the result had many additional hypotheses including planarity and being restricted to the $S$-case.  With some work most of the extra hypotheses dropped away and the current Theorem~\ref{thm main} remained.  It is not an entirely satisfactory theorem as it stands, but it introduces some new ideas to the game and makes nontrivial progress on the completion conjecture for the $c_2$ invariant for almost the first time. I hope that as the first crack in that conjecture it will lead, with the help of others, to a proof of the full conjecture.

\bibliographystyle{plain}
\bibliography{main}

\begin{thebibliography}{10}

\bibitem{Ax}
James Ax.
\newblock Zeroes of polynomials over finite fields.
\newblock {\em Amer. J. Math.}, 86(2):255--261, 1964.

\bibitem{BrBe}
Prakash Belkale and Patrick Brosnan.
\newblock Matroids, motives, and a conjecture of {K}ontsevich.
\newblock {\em Duke Math. J.}, 116(1):147--188, 2003.
\newblock arXiv:math/0012198.

\bibitem{bek}
Spencer Bloch, H\'el\`ene Esnault, and Dirk Kreimer.
\newblock On motives associated to graph polynomials.
\newblock {\em Commun. Math. Phys.}, 267:181--225, 2006.
\newblock arXiv:math/0510011v1 [math.AG].

\bibitem{BrSinform}
David Broadhurst and Oliver Schnetz.
\newblock Algebraic geometry informs perturbative quantum field theory.
\newblock In {\em PoS}, volume LL2014, page 078, 2014.
\newblock arXiv:1409.5570.

\bibitem{bkphi4}
D.J. Broadhurst and D.~Kreimer.
\newblock Knots and numbers in $\phi^4$ theory to 7 loops and beyond.
\newblock {\em Int.J.Mod.Phys.}, C6(519-524), 1995.
\newblock arXiv:hep-ph/9504352.

\bibitem{Brcosmic}
Francis Brown.
\newblock Feynman amplitudes, coaction principle, and cosmic {G}alois group.
\newblock arXiv:1512.06409.

\bibitem{Brbig}
Francis Brown.
\newblock On the periods of some {F}eynman integrals.
\newblock arXiv:0910.0114.

\bibitem{BrS}
Francis Brown and Oliver Schnetz.
\newblock A {K3} in $\phi^4$.
\newblock {\em Duke Math J.}, 161(10):1817--1862, 2012.
\newblock arXiv:1006.4064.

\bibitem{BrS3}
Francis Brown and Oliver Schnetz.
\newblock Modular forms in quantum field theory.
\newblock {\em Communications in Number Theory and Physics}, 7(2):293 -- 325,
  2013.
\newblock arXiv:1304.5342.

\bibitem{BrSY}
Francis Brown, Oliver Schnetz, and Karen Yeats.
\newblock Properties of $c_2$ invariants of {F}eynman graphs.
\newblock {\em Advances in Theoretical and Mathematical Physics},
  18(2):323--362, 2014.
\newblock arXiv:1203.0188.

\bibitem{BrY}
Francis Brown and Karen Yeats.
\newblock Spanning forest polynomials and the transcendental weight of
  {F}eynman graphs.
\newblock {\em Commun. Math. Phys.}, 301(2):357--382, 2011.
\newblock arXiv:0910.5429.

\bibitem{Chai}
Seth Chaiken.
\newblock A combinatorial proof of the all minors matrix tree theorem.
\newblock {\em SIAM J. Alg. Disc. Meth.}, 3(3):319--329, 1982.

\bibitem{CYgrid}
Wesley Chorney and Karen Yeats.
\newblock $c_2$ invariants of recursive families of graphs.
\newblock arXiv:1701.01208.

\bibitem{Cphd}
Iain Crump.
\newblock {\em Graph Invariants with Connections to the Feynman Period in
  $\phi^4$ Theory}.
\newblock PhD thesis, Simon Fraser University, 2017.
\newblock arXiv:1704.06350.

\bibitem{CtreeG}
Richard~L. Cummins.
\newblock Hamilton circuits in tree graphs.
\newblock {\em IEEE transactions on circuit theory}, 13(1), 1966.

\bibitem{Dbook}
Reinhard Diestel.
\newblock {\em Graph Theory}, volume 173 of {\em Graduate Texts in
  Mathematics}.
\newblock Springer, 1997,2000,2009.

\bibitem{Dc2}
D.~Doryn.
\newblock The $c_2$ invariant is invariant.
\newblock arXiv:1312.7271.

\bibitem{D4face}
D.~Doryn.
\newblock Dual graph polynomials and a 4-face formula.
\newblock arXiv:1508.03484.

\bibitem{iz}
Claude Itzykson and Jean-Bernard Zuber.
\newblock {\em Quantum Field Theory}.
\newblock McGraw-Hill, 1980.
\newblock Dover edition 2005.

\bibitem{Lmod}
Adam Logan.
\newblock New realizations of modular forms in {C}alabi-{Y}au threefolds
  arising from $\phi^4$ theory.
\newblock arXiv:1604.04918.

\bibitem{Mmotives}
Matilde Marcolli.
\newblock {\em Feynman Motives}.
\newblock World Scientific, 2010.

\bibitem{Snumbers}
Oliver Schnetz.
\newblock Numbers and functions in quantum field theory.
\newblock arXiv:1606.08598.

\bibitem{Sphi4}
Oliver Schnetz.
\newblock Quantum periods: A census of $\phi^4$-transcendentals.
\newblock {\em Communications in Number Theory and Physics}, 4(1):1--48, 2010.
\newblock arXiv:0801.2856.

\bibitem{SFq}
Oliver Schnetz.
\newblock Quantum field theory over $\mathbb{F}_q$.
\newblock {\em Elec. J. Combin.}, 18, 2011.
\newblock arXiv:0909.0905.

\bibitem{VY}
Aleks Vlasev and Karen Yeats.
\newblock A four-vertex, quadratic, spanning forest polynomial identity.
\newblock {\em Electron. J. Linear Alg.}, 23:923--941, 2012.
\newblock arXiv:1106.2869.

\bibitem{Ycirc}
Karen Yeats.
\newblock A few $c_2$ invariants of circulant graphs.
\newblock {\em Commun. Number Theory Phys.}, 10(1):63--86, 2016.
\newblock arXiv:1507.06974.

\end{thebibliography}

\end{document}